\documentclass{article}
\usepackage[english]{babel}
\usepackage{epsfig}
\usepackage{amsmath}
\usepackage{amsthm}
\usepackage{amssymb}
\usepackage{color}

\newtheorem{theorem}{Theorem}[section]
\newtheorem{definition}[theorem]{Definition}
\newtheorem{lemma}[theorem]{Lemma}
\newtheorem{prop}[theorem]{Proposition}
\newtheorem{corollary}[theorem]{Corollary}
\newtheorem{remark}[theorem]{Remark}
\newtheorem{example}[theorem]{Example}

\newcommand{\hh}{{\mathbb{H}}}
\newcommand{\s}{{\mathbb{S}}}
\newcommand{\cc}{{\mathbb{C}}}
\newcommand{\rr}{{\mathbb{R}}}
\newcommand{\nn}{{\mathbb{N}}}

\newcommand{\an}{{\mathcal{A}}}
\newcommand{\hyp}{\mathcal{H}}

\title{\bf Power series and analyticity over the quaternions}

\author{Graziano Gentili\footnote{Partially supported by GNSAGA of the INdAM and by PRIN ``Propriet\`a geometriche delle variet\`a reali e complesse'' of the MIUR.}  \\ 
\normalsize Dipartimento di Matematica ``U. Dini'', Universit\`a di Firenze \\ 
\normalsize Viale Morgagni 67/A, 50134 Firenze, Italy,  gentili@math.unifi.it \\
\and Caterina Stoppato$^*$\footnote{Partially supported by PRIN ``Geometria Differenziale e Analisi Globale'' of the MIUR.} \\ 
\normalsize Dipartimento di Matematica ``U. Dini'', Universit\`a di Firenze \\ 
\normalsize Viale Morgagni 67/A, 50134 Firenze, Italy,  stoppato@math.unifi.it}

\date{  }
\begin{document}


\maketitle

\begin{abstract}
We study power series and analyticity in the quaternionic setting. We first consider a function $f$ defined as the sum of a power series $\sum_{n \in \nn} q^n a_n$ in its domain of convergence, which is a ball $B(0,R)$ centered at $0$. At each $p \in B(0,R)$, $f$ admits expansions in terms of appropriately defined \emph{regular power series centered at $p$}, $\sum_{n \in \nn} (q-p)^{*n} b_n$. The expansion holds in a ball $\Sigma(p, R-|p|)$ defined with respect to a (non-Euclidean) distance $\sigma$. We thus say that $f$ is \emph{$\sigma$-analytic} in $B(0,R)$. Furthermore, we remark that $\Sigma(p, R-|p|)$ is not always an Euclidean neighborhood of $p$; when it is, we say that $f$ is \emph{quaternionic analytic} at $p$. It turns out that $f$ is quaternionic analytic in a neighborhood of $B(0,R) \cap \rr$ that can be strictly contained in $B(0,R)$. 

We then relate these notions of anayticity to the class of  regular quaternionic functions  introduced in \cite{advances}, and recently extended in \cite{advancesrevised}. Indeed, $\sigma$-analyticity proves equivalent to regularity, in the same way as complex analyticity is equivalent to holomorphy. Hence the theory of regular quaternionic functions, which is quickly developing, reveals a new feature that reminds the nice properties of holomorphic complex  functions. 
\end{abstract}

\section{Introduction}\label{sectionintroduction}


Denote by $\hh$ the skew field of quaternions, obtained by endowing $\rr^4$ with the following multiplication operation: if $1,i,j,k$ denotes the standard basis, define $$i^2 = j^2 = k^2 = -1,$$  $$ ij = -ji = k, jk = -kj = i, ki = -ik = j,$$ let $1$ be the neutral element and extend the operation by distributivity to all quaternions $q = x_0 + x_1 i + x_2 j + x_3 k$.  If we define the \emph{conjugate} of such a $q$ as $\bar q = x_0 - x_1 i - x_2 j - x_3 k$, its \emph{real} and \emph{imaginary part} as $Re(q) = x_0$ and $Im(q) = x_1 i + x_2 j + x_3 k$ and its \emph{modulus} as $|q|= \sqrt{q\bar q} = \sqrt{Re(q)^2 + |Im(q)|^2}$, then the multiplicative inverse of each $q \neq 0$ is computed as 
$$q^{-1} = \frac{\bar q}{|q|^2}.$$

Polynomials and power series with coefficients in the skew field $\mathbb{H}$ are harder to handle than in the usual commutative setting. Indeed, if $q=x_0+i x_1+j x_2+ k x_3$, then it is easy to verify that
\begin{eqnarray}\label{variabilireali}
x_0=\frac{1}{4}(q-i qi-j qj-kqk),\ x_1=\frac{1}{4i}(q-i qi+j qj+kqk), \nonumber \\ 
x_2=\frac{1}{4j}(q+i qi-j qj+kqk),\ x_3=\frac{1}{4k}(q+i qi+j qj-kqk).
\end{eqnarray}
Hence the set of quaternionic polynomials in the quaternonic variable $q$, built as sums of monomials of type $\alpha_0 q \alpha_1 q ... \alpha_{n-1} q \alpha_n$, can be identified with that of $4$-tuples of polynomials in the four real variables $x_0,x_1,x_2,x_3$. This deprives such polynomials of the properties we are familiar with. For instance, some of these polynomials fail to have roots (see \cite{pumplun} or consider the polynomial $i q - q i -1$). For these reasons, the classic algebraic theory restricts to polynomials of the form $a_0 + a_1 q + \dots + a_n q^n$ (see e.g. \cite{lam}). This approach leads to the quaternionic analog of the fundamental theorem of algebra presented in \cite{niven} (later extended in \cite{eilenberg}) and to other nice properties. We adopt an analogous  point of view and consider (polynomials and) power series of the following type.

\begin{definition}
Let $\{a_n\}_{n \in \nn}$ be a sequence in $\hh$. We define 
\begin{equation}\label{centrozero}
f(q) = \sum_{n \in \nn} q^n a_n,
\end{equation}
as the \emph{regular power series} associated to $\{a_n\}_{n \in \nn}$.
\end{definition}

The reason for calling these series ``regular'' will become apparent in section \ref{sectionregularfunctions}. As stated in \cite{advances}, the analog of Abel's theorem holds.

\begin{theorem}\label{abel} 
For every regular power series $f(q) = \sum_{n\in \nn}q^n a_n$ there exists an $R \in [0, +\infty]$, called the \emph{radius of convergence} of $f(q)$, such that the series converges absolutely and uniformly on compact sets in $B(0,R) = \{q \in \hh : |q| <R\}$ and diverges in $\{q \in \hh : |q|>R\}$.
\end{theorem}
As in the complex case, $R$ is defined by the formula 
$$\limsup_{n \to \infty} |a_n|^{1/n} = 1/R.$$
Thanks to theorem \ref{abel}, a regular power series having radius of convergence $R>0$ defines a $C^\infty$ function $f : B(0,R) \to \hh$. Such a function $f$ has very peculiar properties (see \cite{survey}): for instance, its zero set consists of isolated points or isolated 2-spheres. As the ring of polynomials (see \cite{lam}), the set of regular power series converging in $B(0,R)$ is endowed with the usual addition operation $+$ and the multiplicative operation $*$ defined by 
\begin{equation}\label{prodottostar}
\left(\sum_{n \in \nn} q^n a_n \right)*\left( \sum_{n \in \nn} q^n b_n \right) = \sum_{n \in \nn} q^n \sum_{k = 0}^n a_k b_{n-k},
\end{equation}
which we call \emph{regular multiplication}. Notice that these operations define a ring structure, as well as a structure of real algebra (see \cite{survey} and references therein).

Defining a notion of quaternionic analyticity is a delicate issue. A first attempt could be to consider expansions (at $p$) in terms of series of monomials of the type $\alpha_0 (q-p) \alpha_1 (q-p) ... \alpha_{n-1} (q-p) \alpha_n$, as done in \cite{hoshi}. However, due to equations (\ref{variabilireali}), this reduces the problem to that of analyticity in the four real variables $x_0,x_1,x_2,x_3$ (for details, see \cite{sudbery}). A second possibility is to consider series of monomials of the type $(q-p)^n \alpha$, but we are immediately discouraged by the following example.

\begin{example}
The simple monomial $q^2$ does not admit an expansion of the type 
$$q^2 = (q-p)^2 c_2 + (q-p) c_1 + c_0$$
at any point $p \in \hh \setminus \rr$. Indeed, such an equality would imply $q^2 = q^2 -qp - pq + p^2 + q c_1 - pc_1 +c_0$ and $q (c_1-p) = pq$ for all $q \in \hh$. This would yield $c_1-p  = p$ and $qp = pq$ for all $q \in \hh$, which is impossible. The aforementioned expansion can only be performed at points $x \in \rr\subset \hh$, where we get the usual formula $q^2 = (q-x)^2 + (q-x) 2x + x^2$ for all $q \in \hh$.
\end{example}

The difficulty arises in the previous example because $(q-p)^2$ is not a regular polynomial when $p \not \in \rr$. We can solve this problem using \emph{regular powers} of $q-p$ instead. Let $(q-p)^{*n} = (q-p)*\cdots *(q-p)$ denote the $n$-th power of the binomial $q-p$ with respect to the regular multiplication $*$.
Remark that $(q-x)^{*n} = (q-x)^{n}$ for all $x \in \rr$ and for all $n \in \nn$, while $(q-p)^{*n}$ and $(q-p)^{n}$ differ for $p \in \hh \setminus \rr$ and $n \geq 2$. A direct computation leads to the following version of the binomial theorem.

\begin{prop}\label{binomio}
Fix $p \in \hh$ and $m \in \nn$. Then $q^m = \sum_{k = 0}^m (q-p)^{*k} p^{m-k}\binom{m}{k}$.
\end{prop}

We notice that proposition \ref{binomio} provides regular polynomials with expansions at an arbitrary point $p$. It also hints that regular power series $\sum_{n \in \nn} q^n a_n$ converging in $B(0,R)$ might expand at $p \in B(0,R)$ in terms of series of the following type.

\begin{definition}
Let $p \in \hh$ and let $\{a_n\}_{n \in \nn}$ be a sequence in $\hh$. We define
\begin{equation}\label{centrop}
\sum_{n \in \nn} (q-p)^{*n} a_n
\end{equation}
as the \emph{regular power series centered at $p$} associated to $\{a_n\}_{n \in \nn}$.
\end{definition}

We thus propose a possible notion of analyticity.

\begin{definition}\label{defanalyticity}
Let $\Omega$ be an open subset of $\hh$. A function $f : \Omega \to \hh$ is \emph{quaternionic analytic at $p \in \Omega$} if there exists a regular power series $ \sum_{n \in \nn} (q-p)^{*n} a_n$ converging in a neighborhood $U$ of $p$ in $\Omega$ such that $f(q) = \sum_{n \in \nn} (q-p)^{*n} a_n$ for all $q \in U$. We say that $f$ is \emph{quaternionic analytic} if it is quaternionic analytic at all $p \in \Omega$.
\end{definition}

The first question we have to answer is: where does a series of type (\ref{centrop}) converge? If the radius of convergence of the quaternionic power series defined by (\ref{centrozero}) is $R>0$, one might expect the series (\ref{centrop}) to converge for all $q$ in the open Euclidean ball $B(p, R)$ of center $p$ and radius $R$ in $\hh$. Surprisingly, this is not the case. 
We begin in section \ref{sectionestimate} with an estimate of $|(q-p)^{*n}|$ in terms of the function $\sigma$ defined as follows. If, for all $I \in \s = \{q \in \hh : q^2 = -1\}$, we denote $L_I = \rr + I \rr$ the complex line through $0,1$ and $I$ then 
$$
\sigma(q,p) = \left\{
\begin{array}{ll}
|q-p| & \mathrm{if\ } p,q \mathrm{\ lie\ on\ the\ same\ complex\ line\ } L_I\\
\omega(q,p) &  \mathrm{otherwise}
\end{array}
\right.
$$
where
$$
\omega(q,p) = \sqrt{\left[Re(q)-Re(p)\right]^2 + \left[|Im(q)| + |Im(p)|\right]^2}. 
$$
The map $\sigma$ proves to be a distance that is not topologically equivalent to the Euclidean distance (see section \ref{sectionsigma}). Actually, the topology induced by $\sigma$ in $\hh$ is finer than the Euclidean topology.
The sets of convergence of series of type (\ref{centrop}) turn out to be $\sigma$-balls centered at $p$ (see section \ref{sectionconvergence}). The shape of such sets, portrayed in figure 2, is quite curious but coherent with theorem 4.2  in \cite{advancesrevised}.
We then study the quaternionic analyticity of the sum $f$ of a series $f(q) = \sum_{n \in \nn} q^n a_n$ in section \ref{sectionanalyticity}. The main result of this section is the following.
\begin{theorem}
Any function defined by a regular power series $\sum_{n \in \nn}q^n a_n$ which converges in $B = B(0,R)$ is quaternionic analytic in the open set 
$$
\an(B) = \{p \in \hh : 2 |Im(p)| < R - |p|\}.
$$
\end{theorem}

Since $\an(\hh) = \hh$, all quaternionic entire functions are quaternionic analytic in $\hh$. If $B \neq \hh$, the set $\an(B)$ is strictly contained in $B$. Furthermore, the previous theorem is sharp: we exhibit an example of regular power series which converges in $B$ and is quaternionic analytic in $\an(B)$ only. This suggests the definition of a weaker notion of analyticity.

\begin{definition}
Let $\Omega$ be a $\sigma$-open subset of $\hh$. A function $f : \Omega \to \hh$ is \emph{$\sigma$-analytic at $p \in \Omega$} if there exists a regular power series $ \sum_{n \in \nn} (q-p)^{*n} a_n$ converging in a $\sigma$-neighborhood $U$ of $p$ in $\Omega$ such that $f(q) = \sum_{n \in \nn} (q-p)^{*n} a_n$ for all $q \in U$. We say that $f$ is \emph{$\sigma$-analytic} if it is $\sigma$-analytic at all $p \in \Omega$.
\end{definition}

The notion of $\sigma$-analyticity produces the following result (to be compared to the complex case).

\begin{theorem}
A regular power series $\sum_{n \in \nn}q^n a_n$ having radius of convergence $R$ defines a $\sigma$-analytic function on $B(0,R)$.
\end{theorem}

Finally, in section \ref{sectionregularfunctions} we relate our study to that of \emph{slice regular functions}, quaternionic analogs of holomorphic functions that have been recently introduced (see \cite{advancesrevised, advances}). 

\begin{theorem}\label{analyticity2}
A quaternionic function is slice regular in a domain if, and only if, it is $\sigma$-analytic in the same domain. 
\end{theorem}

It is important to point out that, in view of theorem \ref{analyticity2}, the notion of $\sigma$-analyticity, which naturally appears when studying the set of convergence of regular power series centered at $p\in \mathbb{H}$, turns out not to be a new notion. It is in fact equivalent to slice regularity. This equivalence urges a comparison with the classical case of complex holomorphic functions and gives a deeper insight into the theory of slice regular functions. This theory is developing well in different directions and it  has already been extensively studied in a series of papers which show its richness (see for example \cite{cauchy, zeros} and \cite{open}).  Slice regular functions have already proven their interest: for instance, they have applications to the theory of quaternionic linear operators (see \cite{bounded, unbounded, electronic, jga}).  
The achievements of this paper give a new viewpoint, and hence enrich, the theory of slice regular functions.

We conclude the paper by also investigate the stronger notion of quaternionic analyticity. Indeed, we are able to prove that

\begin{theorem}
All  regular functions on a domain $\Omega\subseteq \hh$ are quaternionic analytic in $\an(\Omega)= \{p \in \Omega : 2|Im(p)| < \sigma(p,\partial_\sigma \Omega)\}$, where $\partial_\sigma \Omega$ denotes the boundary of $\Omega$ with respect to the $\sigma$-topology. Moreover, if $\Omega$ intersects the real axis, then $\an(\Omega)$ contains a neighborhood of $\Omega\cap \rr$.
\end{theorem}


\section{Estimates of the regular powers of $q-p$}\label{sectionestimate}

In order to study the convergence of a regular power series $\sum_{n \in \nn} (q-p)^{*n}a_n$ centered at $p \in \hh$, we first estimate the growth of $|(q-p)^{*n}|$ as $n \in \nn$ approaches infinity. Define $P(q) = q-p$ for all $q \in \hh$ and recall a well-known algebraic property of $\hh$: for all $I \in \s = \{q \in \hh : q^2 = -1\}$, the subalgebra $L_I = \rr + I \rr$ of $\hh$ is isomorphic to $\cc$. In particular, if we choose $I$ such that $p \in L_I$ then $p$ commutes with any $z \in L_I$. As a consequence,
$$P^{*n}(z) = (z-p)^n$$
and $|P^{*n}(z)| = |z-p|^n$ for all $z \in L_I$. In order to estimate $|P^{*n}(q)|$ at a generic $q \in \hh$, we begin with the following.

\begin{definition}\label{sigma}
For all $p,q \in \hh$, we define 
\begin{equation}
\omega(q,p) = \sqrt{\left[Re(q)-Re(p)\right]^2 + \left[|Im(q)| + |Im(p)|\right]^2} 
\end{equation}
and
\begin{equation}
\sigma(q,p) = \left\{
\begin{array}{ll}
|q-p| & \mathrm{if\ } p,q \mathrm{\ lie\ on\ the\ same\ complex\ line\ } L_I\\
\omega(q,p) &  \mathrm{otherwise}
\end{array}
\right.
\end{equation}
\end{definition}

\begin{figure}\label{distanza}
  \begin{center}
  \fbox{\includegraphics[height=7cm]{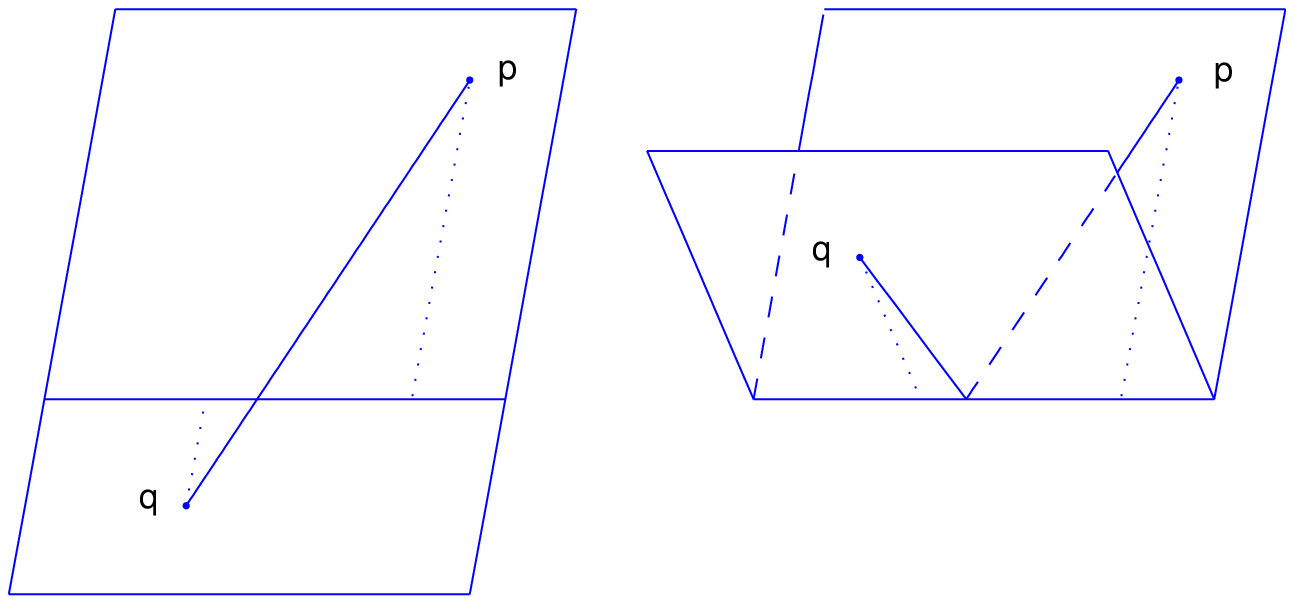}}
\end{center}
  \caption{Examples of ``how to compute'' $\sigma(q,p)$ when $p$ and $q$ lie in the same complex line and when they do not.}
\end{figure}

We then estimate $|(q-p)^{*n}|$ in terms of the map $\sigma$ defined above.

\begin{theorem}\label{modulus}
Fix $p \in \hh$. Then
\begin{equation}
|(q-p)^{*n}| \leq 2 \sigma(q,p)^n
\end{equation}
for all $n \in \nn$. Moreover,
\begin{equation}
\lim_{n \to \infty} |(q-p)^{*n}|^{1/n} = \sigma(q,p).
\end{equation}
\end{theorem}

We will prove theorem \ref{modulus} by means of the following result, presented in \cite{open}. For all $x,y \in \rr$, denote $x+\s y = \{x+Iy : I \in \s\}$ (a 2-sphere if $y \neq 0$, a real singleton $\{x\}$ if $y = 0$).

\begin{theorem}\label{proprietageometrica}
Let $f(q) = \sum_{n \in \nn} q^n a_n$ have radius of convergence $R>0$ and let $x,y \in \rr$ be such that $x+\s y \subset B(0,R)$. There exist $b(x,y), c(x,y) \in \hh$ such that
\begin{equation}
    f(x+Iy) = b(x,y) + I c(x,y)
\end{equation}
for all $I \in \s$.
\end{theorem}

Namely, $b(x,y) = \sum_{n \in \nn} \beta_n a_n$ and $c(x,y) = \sum_{n \in \nn} \gamma_n a_n$ where $\{\beta_n\}_{n \in \nn}$ and $\{\gamma_n\}_{n \in \nn}$ are real sequences such that $(x+Iy)^n = \beta_n + I \gamma_n$ for all $I \in \s$ and all $n \in \nn$.

\begin{corollary}
If the hypotheses of theorem \ref{proprietageometrica} hold, then
\begin{eqnarray}
f(x+Jy) = \frac{f(x+Iy) + f(x-Iy)}{2} + J I \frac{f(x-Iy) - f(x+Iy)}{2} =\\
= \frac{1-JI}{2} f(x+Iy) +  \frac{1+JI}{2} f(x-Iy) \nonumber
\end{eqnarray}
for all $I,J \in \s$.
\end{corollary}

We are now ready to prove the announced result.

\begin{proof}[Proof of theorem \ref{modulus}]
Let $p \in L_I$ and $P(q) = q-p$, consider $P^{*n}(q)$ for $n \in \nn$. As we already mentioned, for all $z = x+Iy \in L_I$, $z$ and $p$ commute so that $P^{*n}(z) = (z-p)^n$ and $|P^{*n}(z)| = |z-p|^n = \sigma(z,p)^n$. 
For $q = x+Jy$ with $J \in \s \setminus \{\pm I\}$ we compute:
$$P^{*n}(q) = \frac{1-JI}{2} P^{*n}(z) +  \frac{1+JI}{2} P^{*n}(\bar z) =$$
$$=  \frac{1-JI}{2} (z-p)^n +  \frac{1+JI}{2} (\bar z-p)^n.$$
If $|z-p|<|\bar z-p|$ then $Q = \frac{z-p}{\bar z-p}$ has modulus $|Q| < 1$ and
$$|P^{*n}(q)| =  \left| \frac{1-JI}{2}Q^n +  \frac{1+JI}{2} \right| |\bar z-p|^n \leq 2 |\bar z-p|^n.$$
Moreover, we observe that $Q^n \to 0$ as $n \to \infty$ and $\frac{1+JI}{2}\neq 0$, so that
$$\lim_{n \to \infty} |P^{*n}(q)|^{1/n} =$$
$$= \lim_{n \to \infty} \left| \frac{1-JI}{2}Q^n +  \frac{1+JI}{2} \right|^{1/n} |\bar z-p| = |\bar z-p|.$$
If $|z-p|\geq|\bar z-p|$ then similar computations show that 
$$|P^{*n}(q)| \leq 2 |z-p|^n,$$
$$\lim_{n \to \infty} |P^{*n}(q)|^{1/n} = |z-p|.$$
Notice that
$$\max\{|z-p|, |\bar z-p|\}^2 =$$
$$= (Re(z)-Re(p))^2 +\max\{|-Im(z)-Im(p)|^2, |Im(z)-Im(p)|^2\} = $$
$$= (Re(z)-Re(p))^2 + (|Im(z)| + |Im(p)|)^2 =$$
$$= (Re(q)-Re(p))^2 + (|Im(q)| + |Im(p)|)^2 = \sigma(q,p)^2.$$
\end{proof}


\section{The map $\sigma$ is a distance}\label{sectionsigma}

Before going back to power series, let us study the properties of $\sigma$ and $\omega$. Notice that $|q-p| \leq \sigma(q,p) \leq \omega(q,p)$ for all $p,q \in \hh$. Moreover:

\begin{prop}
The function $\sigma : \hh \times \hh \to \rr$ is a distance.
\end{prop}

\begin{proof} 
Definition \ref{sigma} yields that $\sigma(q,p) \geq 0$ for all $p, q \in \hh$ and that $\sigma(q,p) = 0$ if and only if $q = p$. It also immediately implies that $\sigma(q,p) = \sigma(p,q)$. The triangle inequality can be proven as follows. Let $L_I$ be a complex line through $p$.
If $q \in L_I$ then for all $o \in \hh$
$$\sigma(q,o)+\sigma(o,p) \geq |q-o|+|o-p| \geq |q-p| = \sigma(q,p)$$
as wanted.
Now suppose $q \in \hh \setminus L_I$, i.e. $q = x+Jy$ with $y>0, J \in \s \setminus \{\pm I\}$. For all $o \in L_I$ we have
$$\sigma(q,o)+\sigma(o,p) = \max\{|x+Iy -o|, |x-Iy - o|\}+|o-p| \geq$$
$$\geq \max\{|x+Iy -p|, |x-Iy - p|\} = \sigma(q,p).$$
For all $o \in L_J$, the triangle inequality is proven as above, by reversing the roles of $q$ and $p$. Finally, for $o \in \hh \setminus (L_I\cup L_J)$ we prove it as follows. Let $z \in L_I$ be such that $\sigma(o,p) = |z-p|$. Since $q,o$ and $z$ lie in three distinct complex lines, we can easily derive from definition \ref{sigma} that $\sigma(q,o) = \sigma(q,z)$. Hence
$$\sigma(q,o)+\sigma(o,p) = \sigma(q,z)+ |z-p| =$$
$$=\max\{|x+Iy -z|, |x-Iy - z|\}+|z-p| \geq$$
$$\geq \max\{|x+Iy -p|, |x-Iy - p|\} = \sigma(q,p),$$
as desired.
\end{proof}

Let us conclude this section studying the $\sigma$-balls 
$$\Sigma(p,R) = \{q \in \hh : \sigma(q,p) < R\}.$$ 
From $|q-p|\leq \sigma(q,p) \leq \omega(q,p)$, we derive 
$$\Omega(p,R)  = \{q \in \hh : \omega(q,p) < R\} \subseteq \Sigma(p,R) \subseteq B(p,R).$$ More precisely: 
\begin{remark}
If $p \in L_I \subset \hh$ and $B_I(p,R) = \{z \in L_I : |z-p| <R\}$ then
\begin{equation}
\Sigma(p,R) = \Omega(p,R) \cup B_I(p,R).
\end{equation}
\end{remark}
We are thus left with studying $\Omega(p,R)$. In \cite{open} we called a set $C \subseteq \hh$ \emph{(axially) symmetric} if all the 2-spheres $x+\s y$ which intersect $C$ are entirely contained in $C$. Remark that a symmetric set is completely determined by its intersection with a complex line $L_I$. 
\begin{remark}\label{chiave}
$\Omega(p,R)$ is the symmetric open set whose intersection with the complex line $L_I$ through $p$ is $B_I(p,R) \cap B_I(\bar p,R)$. In particular $\Omega(p,R)$ is empty when $R \leq |Im(p)|$, it intersects the real axis when $R > |Im(p)|$ and it includes $p$ if and only if $R > 2 |Im(p)|$.
\end{remark}

The fact that $\Omega(p,R)$ is symmetric  is proven observing that, when $p$ is fixed, $\omega(x+Jy,p)$ only depends on $x$ and $y$. Moreover, $\Omega(p,R) \cap L_I = B_I(p,R) \cap B_I(\bar p,R)$ because $\omega(z,p) =\max\{|z-p|,|\bar z -p|\}=\max\{|z-p|,|z -\bar p|\}$ for all $z \in L_I$.

\begin{figure}\label{palle}
  \begin{center}
  \fbox{\includegraphics[height=7cm]{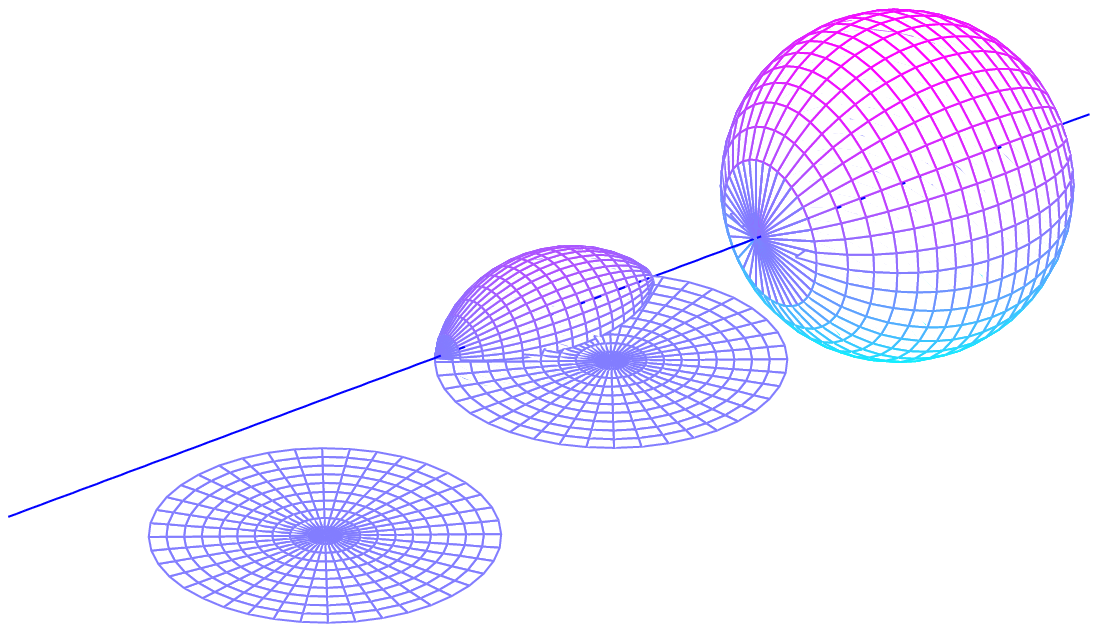}}
\end{center}
  \caption{A view in $\rr+i\rr+j\rr$ of $\sigma$-balls $\Sigma(p,R)$ centered at points $p \in L_i = \rr+i\rr$ and having $|Im(p)| \geq R$, $0 < |Im(p)| <R$ and $Im(p) = 0$, respectively.}
\end{figure}

We conclude this section remarking that 

\begin{remark}\label{chiave1}
$\Omega(p,R)$ is the interior of $\Sigma(p,R)$ and $\overline{\Sigma(p,R)} =  \{q \in \hh : \sigma(q,p) \leq R\}$ (with respect to the Euclidean topology).
\end{remark}


\section{Convergence of regular power series centered at $p$}\label{sectionconvergence}

The estimate given in section \ref{sectionestimate} and the study of the $\sigma$-balls conducted in section \ref{sectionsigma} allow the following result, which will be an important tool in the study of quaternionic analyticity.

\begin{theorem}\label{regularseriesconvergence}
Choose a sequence $\{a_n\}_{n \in \nn} \subset \hh$ and let $R \in [0, +\infty]$ be such that $1/R = \limsup_{n \to \infty} |a_n|^{1/n}$.
For all $p \in \hh$ the series 
\begin{equation}
f(q) = \sum_{n \in \nn} (q-p)^{*n} a_n
\end{equation}
converges absolutely and uniformly on the compact subsets of $\Sigma(p,R)$ and it does not converge at any point of $\hh \setminus \overline{\Sigma(p,R)}$. We call $R$ the \emph{$\sigma$-radius of convergence} of $f(q)$.
\end{theorem}

\begin{proof} In each set $\overline{\Sigma(p,r)}$ with $r<R$, the function series $\sum_{n \in \nn} (q-p)^{*n} a_n$ is dominated by the convergent number series $2 \sum_{n \in \nn} r^n |a_n|$ thanks to theorem \ref{modulus}. By the same theorem $\lim_{n \to \infty} |(q-p)^{*n}|^{1/n} = \sigma(q,p)$, so that
$$\limsup_{n \to \infty}  |(q-p)^{*n} a_n|^{1/n} = \frac{\sigma(q,p)}{R}.$$
Hence the series cannot converge at any point $q$ such that $\sigma(q,p)>R$. 
\end{proof}


\section{Analyticity of regular power series}\label{sectionanalyticity}

The algebra of regular power series centered at $0$ can be endowed with the following derivation.

\begin{definition}\label{derivative}
Let $f(q) = \sum_{n \in \nn} q^n a_n$ be a regular power series. Its \emph{(formal) derivative} is defined as the regular power series $f'(q) = \sum_{n \in \nn} q^n a_{n+1} (n+1)$.
\end{definition}

Remark that $f'(q)$ (which we sometimes denote as $f(q)'$) has the same radius of convergence as $f(q)$. A long computation proves Leibniz's rule:

\begin{prop}
For all regular power series $f(q)$ and $g(q)$ centered at $0$,
\begin{equation}
\left(f(q)*g(q)\right)' = f'(q)*g(q) + f(q)*g'(q).
\end{equation}
\end{prop}

As a consequence, the $n$-th derivative of $f(q) = \sum_{n \in \nn} q^n a_n$, which we denote $f^{(n)}$, can be expressed as:
\begin{equation}\label{derivatives}
f^{(n)}(q) = \sum_{m \in \nn} q^m a_{m+n} \frac{(m+n)!}{m!}.
\end{equation}
In particular, $f^{(n)}(0) = a_n n!$ for all $n \in \nn$, so that
$$f(q) = \sum_{n \in \nn} q^n f^{(n)}(0) \frac{1}{n!}.$$
At this point, it is natural to study the analyticity of $f$. We will need the following estimate, which can be proven as in the complex case.

\begin{lemma}\label{coefficientestimate}
If $f(q)$ has radius of convergence $R>0$ and $p \in B(0,R)$ then
\begin{equation}
\limsup \left|f^{(n)}(p) \frac{1}{n!}\right|^{1/n} \leq \frac{1}{R-|p|}.
\end{equation}
\end{lemma}

We are now ready to exhibit expansions of $f$ at any point $p \in B(0,R)$.

\begin{theorem}\label{sviluppo}
Let $f(q) = \sum_{n \in \nn} q^n a_n$ have radius of convergence $R>0$ and let $p \in B(0,R)$. Then
\begin{equation}\label{equationsviluppo}
    f(q) = \sum_{n \in \nn} (q-p)^{*n} f^{(n)}(p) \frac{1}{n!}
\end{equation}
for all $q \in \Sigma(p,R-|p|)$.
\end{theorem}

\begin{proof}
Thanks to theorem \ref{regularseriesconvergence} and to lemma \ref{coefficientestimate}, the series on the right hand side of equation (\ref{equationsviluppo}) converges absolutely in $\Sigma(p,R-|p|)$. Applying formula (\ref{derivatives}) and lemma \ref{binomio}, we compute
$$\sum_{k \in \nn} (q-p)^{*k} f^{(k)}(p) \frac{1}{k!} = \sum_{k,n \in \nn} (q-p)^{*k} p^n a_{n+k} \frac{(n+k)!}{n!} \frac{1}{k!} =$$
$$= \sum_{m \in \nn} \left [\sum_{k = 0}^m (q-p)^{*k} p^{m-k}\binom{m}{k}\right] a_m = \sum_{m \in \nn} q^m a_m = f(q)$$ for all $q \in \Sigma(p,R-|p|)$.
\end{proof}

As a consequence, according to definition \ref{defanalyticity}, the function $f$ is quaternionic analytic at each $p \in B(0,R)$ such that $p$ lies in the interior of $\Sigma(p, R-|p|)$. Recalling that the interior of $\Sigma(p, R-|p|)$ is the set which we called $\Omega(p,R-|p|)$ and that the latter includes $p$ if and only if $R-|p|> 2 |Im(p)|$, we get the following result.

\begin{theorem}\label{quatanalyticity}
Any function defined by a regular power series $\sum_{n \in \nn}q^n a_n$ which converges in $B = B(0,R)$ is quaternionic analytic in the open set 
\begin{equation}
\an(B) = \{p \in \hh : 2 |Im(p)| < R - |p|\}.
\end{equation}
\end{theorem}
We remark that $\an(B)\subset B$ is the (open) region bounded by the hypersurface consisting of all points $p \in B$ such that $x = Re(p)$ and $y=|Im(p)|$ verify
\begin{equation}\label{primaiperbole}
x^2 -3 \left(y- \frac{2}{3}R \right)^2 + \frac{R^2}{3} = 0.
\end{equation}
In other words, if for any $I\in \s$ we consider the arc of hyperbola $\hyp(I)=\{ x+Iy \in L_I : 0\leq y < R, \ x, y\ \mathrm{verify\ (\ref{primaiperbole})} \},$ then the boundary of $\an(B)$ is the hypersurface of revolution generated  rotating $\hyp(I)$ around the real axis as follows:
$$
\partial \an(B) = \bigcup_{J\in \s} \hyp(J).
$$
The previous result is quite surprising, if compared to the analogous result for complex power series. Even more surprisingly, it is sharp: the following example proves that some regular power series with radius of convergence $R$ are quaternionic analytic in $\an(B)$ only. Recall that, as explained in \cite{remmert}, the complex series $\sum_{n \in \nn} z^{2^n}$ converges in the open unit disc of $\cc$ and it does not extend to a holomorphic function near any point of the boundary.

\begin{example}
The sum of the quaternionic series $f(q) = \sum_{n \in \nn} q^{2^n}$, which converges in $B = B(0,1)$, is quaternionic analytic in $\an(B)$ only. Indeed, suppose $f$ were quaternionic analytic at $p \in B \setminus \an(B)$ (i.e. at $p \in B$ with $2 |Im(p)| \geq 1 -|p|$). There would exist a $\sigma$-ball $\Sigma(p,r)$ of radius $r > 2 |Im(p)|$ and a regular power series $\sum_{n \in \nn} (q-p)^{*n} a_n$ converging in $\Sigma(p,r)$ and coinciding with $f(q)$ for all $q \in B(0,1) \cap \Sigma(p,r)$. The coefficients $a_n$ would have to lie in the complex line $L_I$ through $p$, because $f(B_I(0,1)) \subseteq L_I$ (using a slight variation of corollary 2.8 in \cite{advances}). Restricting to $L_I = \rr+I\rr \simeq \cc$, we would have $f(z) = \sum_{n \in \nn} (z-p)^n a_n$ for all $z \in B_I(0,1) \cap B_I(p,r)$, with $B_I(p,r)$ not contained in the unit disc $B_I(0,1)$ because $r>2 |Im(p)|\geq1-|p|$. The sum of the series $\sum_{n \in \nn} (z-p)^n a_n$ in $B_I(p,R)$ would thus extend $\sum_{n \in \nn} z^{2^n}$ near some point of the boundary of $B_I(0,1)$, which is impossible.
\end{example}

The only case when we immediately conclude quaternionic analyticity on the whole domain is the following.

\begin{theorem}
A \emph{quaternionic entire} function, i.e. the sum of a power series $\sum_{n \in \nn} q^n a_n$ converging in $\hh$, is always quaternionic analytic in $\hh$.
\end{theorem}

The peculiar results we have proven depend on the fact that the sets of convergence of regular power series are $\sigma$-balls, which are not open in the Euclidean topology. We are thus encouraged to define a weaker notion of analyticity in terms of the topology induced in $\hh$ by the distance $\sigma$, a topology that is finer than the Euclidean.

\begin{definition}
Let $\Omega$ be a $\sigma$-open subset of $\hh$. A function $f : \Omega \to \hh$ is \emph{$\sigma$-analytic} at $p \in \Omega$ if there exist $\sum_{n \in \nn} (q-p)^{*n} a_n$ and $R>0$ such that $f(q) = \sum_{n \in \nn} (q-p)^{*n} a_n$ in $\Sigma(p,R) \cap \Omega$. We say that $f$ is \emph{$\sigma$-analytic} if it is $\sigma$-analytic at all $p \in \Omega$.
\end{definition}

We remark that in the previous definition (and in the rest of the paper) the series we consider are still meant to converge with respect to the usual Euclidean norm. We only make use of the $\sigma$-topology when dealing with subsets of $\hh$ (typically, with the domains of definition of regular functions) and only when explicitly stated. 

We immediately deduce from theorem \ref{sviluppo} the following result. We notice the complete analogy to the complex case.

\begin{theorem}
A regular power series $\sum_{n \in \nn}q^n a_n$ having radius of convergence $R$ defines a $\sigma$-analytic function on $B(0,R)$.
\end{theorem}

At this point, it would be natural to inquire about the quaternionic and $\sigma$-analyticity of sums of regular power series $\sum_{n \in \nn} (q-p)^{*n} a_n$ centered at an arbitrary point $p \in \hh$ and converging in a $\sigma$-ball $\Sigma(p,R)$. Instead of employing the same techniques we applied in this first enquire, we postpone this problem to the next section where such series are studied as part of a larger class of quaternionic functions.


\section{Analyticity of quaternionic regular functions}\label{sectionregularfunctions}

Besides its independent interest, the study we conducted in the previous sections is motivated by the theory of quaternion-valued functions of one quaternionic variable introduced in \cite{advances} and developed in subsequent papers (see the survey \cite{survey} and the references therein). The theory, which is also the basis for a new functional calculus in a non commutative setting (see \cite{electronic,jfa}), relies upon a definition of regularity for quaternionic functions inspired by C. G. Cullen \cite{cullen} (see also \cite{deleo}).

\begin{definition}\label{definition}
Let $\Omega$ be a domain in $\hh$ and let $f : \Omega \to \hh$ be a function. For all $I \in \s$, we denote $L_I = \rr + I \rr$, $\Omega_I = \Omega \cap L_I$ and $f_I = f_{|_{\Omega_I}}$. 
The function $f$ is called \textnormal{slice regular} if, for all $I \in \s$, the restriction $f_I$ is holomorphic, i.e. the function $\bar \partial_I f : \Omega_I \to \hh$ defined by
\begin{equation}
\bar \partial_I f (x+Iy) = \frac{1}{2} \left( \frac{\partial}{\partial x}+I\frac{\partial}{\partial y} \right) f_I (x+Iy)
\end{equation}
vanishes identically.
\end{definition}

From now on, we will refer to slice regular functions as regular functions \emph{tout court}. As proven in \cite{advances}, the class of regular functions includes the sums of the series $\sum_{n\in \nn} q^n a_n$ in their balls of convergence $B(0,R)$. This explains why we called these series ``regular''. Moreover, the following can be proven by direct computation.

\begin{theorem}\label{regularseriesregularity}
Let $p \in \hh$ and let $f(q) = \sum_{n \in \nn} (q-p)^{*n} a_n$ have $\sigma$-radius of convergence $R$. If $\Omega(p,R) \neq \emptyset$ then $f : \Omega(p,R) \to \hh$ is a regular function.
\end{theorem}

Regular functions are not always as nice as the aforementioned. Indeed, if the domain $\Omega$ is not carefully chosen then a regular function $f : \Omega \to \hh$ does not even need to be continuous.

\begin{example}\label{discontinuous}
Choose $I \in \s$ and define $f: \hh \setminus \rr \to \hh$ as follows:
$$f(q) = \left\{ 
\begin{array}{ll}
0 \ \mathrm{if} \ q \in \hh \setminus L_I\\
1 \ \mathrm{if} \ q \in L_I \setminus \rr
\end{array}
\right.$$
This function is clearly regular, but not continuous.
\end{example}

Properties of the domains of definition of regular functions which prevent such pathologies have been identified in \cite{advancesrevised}.

\begin{definition}
Let $\Omega$ be a domain in $\hh$, intersecting the real axis. If $\Omega_I = \Omega \cap L_I$ is a domain in $L_I \simeq \cc$ for all $I \in \s$ then we say that $\Omega$ is a \textnormal{slice domain}.
\end{definition}

The following result holds for regular functions on slice domains.

\begin{theorem}[Identity principle] \label{identity}
Let $\Omega$ be a slice domain and let $f,g : \Omega \to \hh$ be regular. Suppose that $f$ and $g$ coincide on a subset $C$ of $\Omega_I$, for some $I \in \s$. If $C$ has an accumulation point in $\Omega_I$, then $f \equiv g$ in $\Omega$.
\end{theorem}

Analogs of many other classical results in complex analysis are proven in \cite{advancesrevised} for regular functions on slice domains: the integral representation formula, the mean value property, the maximum modulus principle, Cauchy's estimates, Liouville's theorem, Morera's theorem, Schwarz's lemma. 

Adding the condition of axial symmetry (defined in section \ref{sectionsigma}) leads to even finer properties. Indeed, theorem \ref{proprietageometrica} extends to all regular functions on a symmetric slice domain $\Omega$, yielding that such functions are $C^{\infty}(\Omega)$. The condition of symmetry is not restrictive when studying regular functions on slice domains, because of the following (surprising) theorem. Define the \emph{symmetric completion} of each set $T\subseteq \hh$ as the symmetric set $$\widetilde{T} = \bigcup_{x+Iy \in T} \left(x+\s y\right).$$

\begin{theorem}[Extension theorem]
Let $\Omega \subseteq \hh$ be a slice domain and let $f : \Omega \to \hh$ be a regular function. Then $\widetilde{\Omega}$ is a symmetric slice domain and $f$ extends to a unique regular function $\tilde f : \widetilde{\Omega} \to \hh$.
\end{theorem}

In addition to its intrinsic meaning, this result (proven in \cite{advancesrevised}) allows the definition of an operation $*$, called \emph{regular multiplication}, on the set $\mathcal{R}(\Omega)$ of regular functions on a (symmetric) slice domain $\Omega$. It turns out that $(\mathcal{R}(\Omega), + , *)$ is a (non-commutative) associative real algebra and this algebraic structure is strictly related to the distribution of the zeros of regular functions. Furthermore, the multiplicative inverse $f^{-*}$ of any $f\not \equiv 0$ in $\mathcal{R}(\Omega)$ is computed in \cite{advancesrevised}, extending results proven in \cite{open}. Finally, the minimum modulus principle and the open mapping theorem presented in \cite{open} extend to all regular functions on (symmetric) slice domains.

We warn the reader: several results that we will prove in this section are stated for regular functions defined on generic domains $\Omega$ in $\hh$, but they only acquire full significance when $\Omega$ is a slice domain.

The study conducted in the previous sections allows us to provide regular functions on symmetric slice domains with power series expansions.

\begin{definition}
Let $\Omega$ be a domain in $\hh$ and let $f : \Omega \to \hh$ be a regular function. Its \textnormal{slice derivative} is the regular function $f'$ which equals
\begin{equation}
\partial_I f (x+Iy) = \frac{1}{2} \left( \frac{\partial}{\partial x}-I\frac{\partial}{\partial y} \right) f_I (x+Iy)
\end{equation}
on $\Omega_I$, for all $I \in \s$.
\end{definition}

We may denote the slice derivative of $f$ by $f'$ and its $n$-th slice derivative by $f^{(n)}$ without causing any confusion, since the slice derivative of $\sum_{n \in \nn} q^n a_n$ coincides with the formal derivative $\sum_{n \in \nn} q^n a_{n+1}(n+1)$. With this notation, we rephrase a result proven in \cite{advancesrevised} as follows.

\begin{theorem}\label{sviluppoinunreale}
Let $f$ be a regular function on a domain $\Omega \subseteq \hh$ and let $x \in \Omega \cap \rr$. In each Euclidean ball $B(x,R)$ contained in $\Omega$ the function $f$ expands as
\begin{equation}
f(q) = \sum_{n \in \nn} (q-x)^n f^{(n)}(x)\frac{1}{n!}.
\end{equation}
\end{theorem}

We are now able to extend the previous result as follows.

\begin{theorem}\label{expansion}
Let $f$ be a regular function on a domain $\Omega \subseteq \hh$ and let $p \in \Omega$. In each $\sigma$-ball $\Sigma(p,R)$ contained in $\Omega$ the function $f$ expands as
\begin{equation}\label{quaternionicexpansion}
f(q) = \sum_{n \in \nn} (q-p)^{*n} f^{(n)}(p)\frac{1}{n!}.
\end{equation}
\end{theorem}

\begin{proof}
By construction $\Omega_I \supseteq \Sigma(p,R) \cap L_I = B_I(p,R)$. By the properties of holomorphic functions of one complex variable, $f_I(z)$ expands as $\sum_{n \in \nn} (z-p)^n f^{(n)}(p)\frac{1}{n!}$ in $B_I(p,R)$. We conclude that the series in equation (\ref{quaternionicexpansion}) converges in $\Sigma(p,R)$. If $\Omega(p,R)$ is empty then the assertion is proved. Otherwise, the series in equation (\ref{quaternionicexpansion}) defines a regular function $g: \Omega(p,R) \to \hh$. Since $f_I \equiv g_I$ in $B_I(p,R)$, the identity principle \ref{identity} allows us to conclude that $f$ and $g$ coincide in $\Omega(p,R)$ (hence in $\Sigma(p,R)$, as desired).
\end{proof}

\begin{corollary}
A quaternionic function is slice regular in a domain if, and only if, it is $\sigma$-analytic in the same domain. 
\end{corollary}

The previous corollary recalls the complex theory, although its significance is quite different because of the properties of the topology involved in this case. Indeed, example \ref{discontinuous} proves that if the domain of definition is ill-chosen then a $\sigma$-analytic function need not be continuous. On the other hand, we might inquire about the quaternionic analyticity of a regular function (which implies $C^\infty$ continuity). It turns out that the set of quaternionic analyticity is far more complicated, as we are about to see. For all $T \subseteq \hh$, we denote $\partial_\sigma T$ its $\sigma$-boundary, we define
$$\an(T)= \{p \in T : 2|Im(p)| < \sigma(p,\partial_\sigma T)\}$$
and obtain the following  direct consequence of remarks \ref{chiave}, \ref{chiave1} and theorems \ref{regularseriesconvergence}, \ref{expansion}.

\begin{corollary}\label{seidieci}
All  regular functions on a domain $\Omega\subseteq \hh$ are quaternionic analytic in $\an(\Omega)$.
\end{corollary}

Remark that if $\Omega$ does not intersect the real axis $\rr$ then $\an(\Omega)$ is empty.

We now consider the case of a regular power series $f(q)$ centered at a point $p\in \hh$ and converging in $\Sigma(p,R) = \Omega(p,R) \cup B_I(p,R)$. By the previous corollary, its sum $f$ is quaternionic analytic in $\an(\Omega(p,R))$. Suppose $p=x_0+I y_0$ with $I \in \s$ and $x_0,y_0, R \in \rr$ such that $0<y_0<R$. By direct computation, $\an(\Omega(p,R))\subset \Omega(p,R)$ is the (open) region bounded by the hypersurface of points $s \in \Omega(p,R)$ such that $x = Re(s)$ and $y = |Im(s)|$ solve the equation
\begin{equation}\label{aperto}
(x-x_0)^2-3\left(y-\frac{y_0+2R}{3}\right)^2+\frac{(2y_0+R)^2}{3}=0,
\end{equation}
i.e. the hypersurface of revolution $\bigcup_{J \in \s} \hyp(J)$ where $\hyp(J)$ is the arc of hyperbola 
$$\hyp(J) = \{x+Jy : 0 \leq y \leq R-y_0, \ x,y\ \mathrm{verify\ (\ref{aperto})} \}$$ 
for all $J \in \s$. Furthermore, the following theorem proves that $f$ is also quaternionic analytic in $\an(B_I(p,R))$, which is the plane region in $L_I$ lying between $\hyp(-I)$ and the arc of hyperbola $\mathcal{K}(I)$ consisting of all points $x+Iy$ with $0 \leq y \leq R-y_0$ and
\begin{equation}
(x-x_0)^2-3\left(y+\frac{y_0-2R}{3}\right)^2+\frac{(2y_0-R)^2}{3}=0.
\end{equation}
We notice that $\hyp(I)$ always lies between $\hyp(-I)$ and $\mathcal{K}(I)$ in $L_I$, so that $\an(B_I(p,R))$ always includes $\an(\Omega(p,R))\cap L_I$ (which is the plane region lying between $\hyp(-I)$ and $\hyp(I)$). We also notice that $\mathcal{K}(I)$ always lies in $B_I(p,R) \cap B_I(\bar p, R)$, so that $\an(B_I(p,R))$ is always included in $\Omega(p,R)$.

\begin{figure}\label{hyp}
  \begin{center}
  \fbox{\includegraphics[height=7cm]{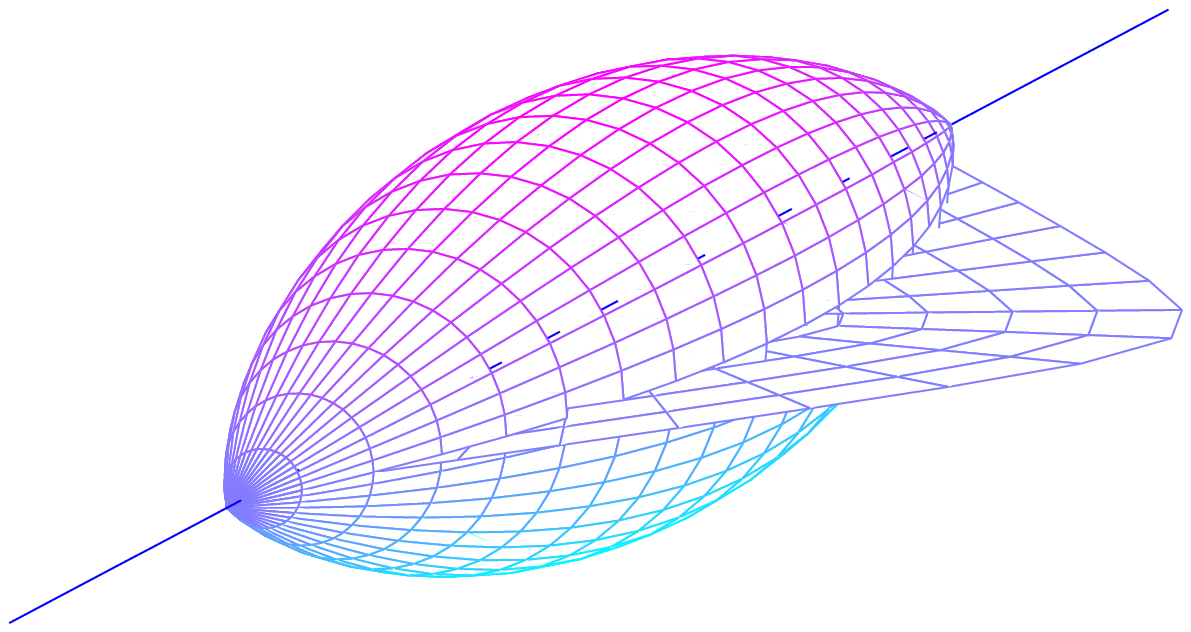}}
\end{center}
  \caption{A view in $\rr+i\rr+j\rr$ of $\an(\Sigma(p,R)) = \an(\Omega(p,R)) \cup \an(B_i(p,R))$, supposing $p \in L_i = \rr+i\rr$.}
\end{figure}

\begin{theorem}\label{anseries}
Choose $p \in \hh$, let $L_I$ be a complex line through $p$ and consider a regular power series $f(q) = \sum_{n \in \nn}(q-p)^{*n} a_n$ having $\sigma$-radius of convergence $R$. If $\Omega(p,R)$ is not empty then $f$ is quaternionic analytic at each point of
\begin{equation}\label{ansigma}
\an(\Sigma(p,R)) = \an(\Omega(p,R)) \cup \an(B_I(p,R)).
\end{equation}
The latter is a subset of $\Omega(p,R)$ and it is not open, unless $p \in \rr$.
\end{theorem}

\begin{proof}
We first prove equality (\ref{ansigma}). Let $\Sigma = \Sigma(p,R), \Omega = \Omega(p,R)$ and $B_I = B_I(p,R)$. For all $s \in \Sigma$, we remark that $\sigma(s,\partial_\sigma \Sigma) = R - \sigma(s,p)$. For all $s \in \Sigma \setminus L_I  = \Omega\setminus L_I$ we get $\sigma(s,\partial_\sigma \Sigma) = R -  \omega(s,p)$ and the latter coincides with $\sigma(s,\partial_\sigma \Omega)$. This proves that 
$$\an(\Omega) \setminus L_I = \an(\Sigma) \setminus L_I.$$ 
As for $s \in L_I$, we get $\sigma(s,\partial_\sigma \Sigma) = R -  |s - p| = \sigma(s, \partial_\sigma B_I)$ so that 
$$\an(\Sigma) \cap L_I = \an(B_I).$$ 
We conclude this first part of the proof recalling that $\an(B_I)$ always includes $\an(\Omega)\cap L_I$, so that
$$\left(\an(\Omega) \setminus L_I\right) \cup \an(B_I) = \an(\Omega)  \cup \an(B_I).$$
Now let us prove that $f$ is quaternionic analytic at all $s \in \an(\Sigma)$. We already noticed that it is quaternionic analytic at all $s \in \an(\Omega)$, due to corollary \ref{seidieci}. Finally, we prove the quaternionic analyticity at all $s \in \an(B_I) \subset \Omega$ as follows. By the properties of holomorphic complex functions, the restriction $f_I(z)$ expands as $\sum_{n \in \nn} (z-s)^{n} f^{(n)}(s)\frac{1}{n!}$ for all $z \in B_I(s, R-|s-p|)$. Hence the expansion $f(q) = \sum_{n \in \nn} (q-s)^{*n} f^{(n)}(s)\frac{1}{n!}$ at $s$ is valid not only in $\Sigma(s,R - \omega(s,p))$ as guaranteed by theorem \ref{expansion}, but in the whole 
$\Sigma(s, R - |s-p|)$. Since $s \in \an(B_I)$ implies $R - |s-p| > 2 |Im(s)|$, the $\sigma$-ball $\Sigma(s, R - |s-p|)$ is an Euclidean neighborhood of $s$ as desired.
\end{proof}

We conclude with an application of theorem \ref{anseries}.

\begin{prop}
For any domain $\Omega \subseteq \hh$ intersecting the real axis, the set $\an(\Omega)$ contains a symmetric neighborhood of $\Omega\cap \rr$.
\end{prop}

\begin{proof}
For all $x \in \Omega\cap \rr$, there exists an $\varepsilon_x >0$ such that $B(x,\varepsilon_x)\subseteq \Omega$. Then
$$\an(\Omega) = \{p \in \Omega : 2|Im(p)| < \sigma(p,\partial_\sigma \Omega)\} \supseteq $$
$$\supseteq  \{p \in B(x,\varepsilon_x) : 2|Im(p)| < \varepsilon_x  - |p|\} = \an(B(x,\varepsilon_x)).$$
The union $\an(B(x,\varepsilon_x))$ for $x \in \Omega\cap \rr$ is a symmetric slice domain containing $\Omega\cap \rr$.
\end{proof}


\bibliography{PowerSeries}

\bibliographystyle{abbrv}


\end{document}